\documentclass[11pt]{amsart}

\usepackage{geometry}
\usepackage{amsmath,amssymb,amsthm,mathtools}
\usepackage[colorlinks=true,linkcolor=blue,citecolor=blue,urlcolor=blue]{hyperref}
\usepackage{tikz}
\usetikzlibrary{calc}

\newtheorem{theorem}{Theorem}

\newtheorem{corollary}[theorem]{Corollary}
\newtheorem{lemma}[theorem]{Lemma}
\newtheorem{definition}[theorem]{Definition}
\newtheorem{remark}[theorem]{Remark}
\newtheorem{example}[theorem]{Example}

\newcommand{\Ed}{\operatorname{Ed}}

\title{Polynomiality of the Generalized Verschiebung Degree}
\author{Siqing Zhang}
\date{}
\begin{document}
\maketitle
\thispagestyle{empty}

\begin{abstract}
For a general curve in positive characteristic, taking the Frobenius pullback induces a generically finite rational map $V$ on the moduli space of rank 2 vector bundles with trivial determinant.
Recently, Kondo--Wakabayashi show that the generic degree of $V$, considered as a function on the characteristic of the base field, is a quasi-polynomial.
In this paper, we show that this quasi-polynomial is indeed a polynomial, and we write out this polynomial explicitly.
\end{abstract}

\section{Introduction}

Let $X$ be a smooth curve in characteristic $p>2$.
Let $F_X: X\to X'$ be the relative Frobenius morphism over $k$.
Let $M_2(X)$ and $M_2(X')$ be the moduli space of stable rank 2 vector bundles of trivial determinant on $X$ and $X'$ respectively.
Taking the Frobenius pullback $E\mapsto F_X^*E$ defines a generically finite rational map $V: M_2(X')\dashrightarrow M_2(X)$, called the \textit{generalized Verschiebung} \cite[Theorem A.6]{Oss1}.
The behavior of this rational map has been studied extensively; see, e.g.
\cite{Gie,LP1,LP2,Oss1,JRXY,LanPa,JoPa,HoWa}.

When $X$ is a general curve of genus $g=2$, the generic degree of $V$ is $\deg(V)=\frac{p^3+2p}{3}$, as shown by Osserman \cite[Theorem 1.3]{Oss1} and Lange--Pauly
\cite[Corollary]{LanPa}.
For higher genus, Kondo--Wakabayashi recently expressed the generic degree in
terms of certain edge-numberings on a graph \cite[Theorem 5.3]{KW}.  
They also show that
$\deg(V)$, as a function of the characteristic $p$, is given by a
quasi-polynomial in $p$ of degree \(3g-3\) \cite[Theorem 5.5]{KW}.
They then use direct computation to show that this quasi-polynomial is a polynomial for $g=3$ \cite[Theorem C]{KW}.

The first result of this paper shows that the quasi-polynomial $\deg(V)$ is always a polynomial in $p$ for arbitrary genus $g\geq 2$.

\begin{theorem}[Polynomiality]\label{thm: polynomiality}
    Let $X$ be a general curve of genus $g\geq 2$ over an algebraically closed field of characteristic $p>2$.
    Let $R_g(t)\in \mathbb{Q}[t]$ be the following polynomial:
    \[R_g(t)=
1-\frac{1}{2^g}+\frac{1}{2^g}
\sum_{r=0}^{g-1}
\frac{(-1)^{r-1} 2^{4r}B_{2r}}{(2r)!}
\left[z^{-2r}\right]
\csc^{2g-2}(z)\,
t^r ,\]
where $B_{2r}$ denotes the $2r$-th Bernoulli number, and $\left[z^{-2r}\right]
\csc^{2g-2}(z)$ denotes the coefficient of $z^{-2r}$ in the Laurent expansion of $\csc^{2g-2}(z)$ at $z=0$.
Then 
\[\deg(V)=p^{g-1} R_g(p^2).\]
\end{theorem}

The proof relies on a new combinatorial result in Wakabayashi's \textit{enumerative geometry of
dormant opers}, see \cite{WakSpin} and
\cite[\S1]{WakTQFT} for an illuminating introduction of the related ideas.

Namely, given a finite connected trivalent graph $G$ of genus $g$, and two natural numbers $P\geq 3$ and $N\geq 1$, Kondo--Wakabayashi define a finite set  $\Ed_{P,N,G}$ of certain edge-numberings on $G$ and show that 
\begin{equation}\label{eqn: deg}
    \deg(V)=\frac{|\Ed_{p,2,G}|}{|\Ed_{p,1,G}|}.
\end{equation}
The set $\Ed_{P,N,G}$ is recalled in Definition \ref{def: Ed}.
When $N=1$, the set $\Ed_{P,1,G}$, implicit already in works of S. Mochizuki, is given by the lattice points in a dilation of a polytope defined by Liu-Osserman in \cite[Definition 2.3]{LO}. 
Our combinatorial result reduces the level-$N$ edge-numbering sets to the level-1 counterparts.

\begin{theorem}[Level-reduction]\label{thm: level-reduction}
    Let $P\geq 3$ be an odd integer and $N\geq 1$ be an integer, and let $G$ be a finite trivalent graph.
    Let $\Ed_{P,N,G}$ be the set of $(P,N)$-edge-numberings of $G$ as in Definition \ref{def: Ed}.
    Then, there is a natural bijection
    \[\Ed_{P,N,G}\cong \Ed_{P,1,G}\times \Ed_{2P,1,G}^{N-1}.\]
\end{theorem}

\subsubsection*{Combinatorial background}

In his work on $p$-adic Teichm\"uller Theory \cite{Moc}, S. Mochizuki linked the nowadays so-called \textit{dormant opers} to the uniformization problem of $p$-adic curves, and showed that the number of dormant opers in characteristic $p$ is given by $|\Ed_{p,1,G}|$, see Liu-Osserman's summary in \cite[Theorem 3.9]{LO}.

By Ehrhart Theory \cite[Theorem 4.6.8]{StanleyEC1}, the number $|\Ed_{P,1,G}|$ of elements of $\Ed_{P,1,G}$, considered as a function of $P$, is a priori a \textit{quasi-polynomial}.
That is, a function $f:\mathbb{Z}\to \mathbb{R}$ of the form $f(n)=\sum_{i=0}^M c_i(n) n^i$, where each $c_i(n)$ is a periodic function for $n$.
Liu-Osserman show the polynomiality of $|\Ed_{P,1,G}|$ for odd $P$ in \cite[Corollary 3.6]{LO}.
In \cite[Theorem 5]{FdPAR}, the polynomiality of $|\Ed_{P,1,G}|$ for even $P$ is shown by Fernandes--de Pina--Ramírez Alfonsín--Robins.

The derivation of the explicit formulas for the odd and even polynomials have a different flavor.
The polynomial for $|\Ed_{P,1,G}|$ when $P$ is odd is given by Wakabayashi in \cite[Theorem A]{WakFormula} as 
\[
        |\Ed_{P,1,G}|
        =
        \frac{P^{g-1}}{2^{2g-1}}
        \sum_{\theta=1}^{P-1}
        \csc^{2g-2}\!\left(\frac{\pi\theta}{P}\right)
        \qquad (P\ \mathrm{odd}).
\]
His proof passes through Quot schemes and Gromov-Witten Theory.
In the proof of Theorem \ref{thm: polynomiality} (see \eqref{eqn: deg trig}), we give an elementary derivation of the corresponding formula for even $P$:
\[
        |\Ed_{P,1,G}|
        =
        \left(\frac P2\right)^{g-1}
        +
        \left(\frac P4\right)^{g-1}
        \sum_{\theta=1}^{P/2-1}
        \csc^{2g-2}\!\left(\frac{\pi\theta}{P}\right)
        \qquad (P\ \mathrm{even}).
\]

\begin{remark}
    Theorem \ref{thm: level-reduction} and \eqref{eqn: deg} give us $\deg(V)=|\Ed_{2p,1,G}|$.
Therefore, we have the following heuristic observation, which we find amusing:
The generic degree $\deg(V)$ in characteristic $p$ agrees with the number of dormant opers over a nonexistent ``field of characteristic $2p$".
\end{remark}

\subsubsection*{Formalization}
The combinatorial cores of this paper, namely, the proofs of Lemma \ref{lm: local fac}, Theorem \ref{thm: level-reduction}, and Lemma \ref{lm: trig}, have been formalized and verified in Lean 4, see \cite{VerschiebungLean}.

\subsubsection*{Outline}
We first prove Theorem \ref{thm: level-reduction} in \S\ref{sec: proof theorem 2}.
We then prove Theorem \ref{thm: polynomiality} in \S\ref{sec: polynomiality}.
Finally, we briefly recall the notion of dormant opers and record the related Corollary \ref{cor: counts} in \S\ref{sec: oper}.

\section{Proof of Theorem \ref{thm: level-reduction}}\label{sec: proof theorem 2}

The set $\Ed_{P,N,G}$ used by Kondo--Wakabayashi in \cite{KW} is given by certain numberings of the edges of $G$, such that, at each vertex of $G$, the numbers are constrained by a set ${}^\dagger C_N(P)$.
In \S\ref{subsec: setup}, we recall the set ${}^\dagger C_N(P)$ and prove a relation among the sets when we change $N$'s and $P$'s.
In \S\ref{subsec: global}, we recall the set $\Ed_{P,N,G}$ and use \S\ref{subsec: setup} to prove Theorem \ref{thm: level-reduction}.

\subsection{The set ${}^\dagger C_N(P)$ as in \cite[\S5.1]{KW}}\label{subsec: setup}\;

For integers $P\geq 2$ and $N\geq 1$, let ${}^\dagger C_N(P)$ denote the set of triples $(a_1,a_2,a_3)\in \mathbb{Z}_{\geq 0}^3$ such that 
\begin{equation}\label{eq: upper}
    a_1+a_2+a_3\leq P^N-2,\quad |a_2-a_3|\leq a_1\leq a_2+a_3,
\end{equation}
and such that, for every $1\leq M\leq  N$, if we let $[a_i]_M$ be the remainder of $a_i$ modulo $P^M$, then there are choices
\begin{equation}\label{eq: am}
    a_{i,M}\in \{[a_i]_M,\, P^M-1-[a_i]_M\}
\end{equation}
such that 
\begin{equation}\label{eq: lower}
a_{1,M}+a_{2,M}+a_{3,M}\leq P^M-2,\quad |a_{2,M}-a_{3,M}|\leq a_{1,M}\leq a_{2,M}+a_{3,M}.
\end{equation}
When $N=1$ or when $M=N$, we can choose $a_{i,N}=a_i$ and the condition \eqref{eq: lower} is redundant.

Note that the definition of ${}^\dagger C_N(P)$ is independent of the ordering of the triple $(a_1,a_2,a_3)$.

We record a lemma that will be useful later on. The proof is straightforward.

\begin{lemma}\label{eqn: four ineq}
    The condition \eqref{eq: upper} is equivalent to the following four inequalities
\begin{align*}
    2P^N -(2a_1+1)-(2a_2+1)-(2a_3+1) &>0,\\
    -(2a_i+1)+\sum_{j\neq i}(2a_j+1)&>0, \quad i=1,2,3.
\end{align*}
\end{lemma}

\begin{example}
${}^\dagger C_1(3)$ is the singleton $\{(0,0,0)\}$.

${}^\dagger C_2(3)$ is 
\[\{(0,0,0)\}\cup S_3\cdot (2,2,0)\cup \{(2,2,2)\} \cup S_3\cdot (3,2,2)\cup S_3\cdot (3,3,0),\]
where $S_3\cdot x$ means the orbit of $x$ under the action of $S_3$ by permuting the coordinates.

${}^\dagger C_1(6)$ is 
\[\{(0,0,0)\}\cup S_3\cdot (1,1,0)\cup \{(1,1,1)\} \cup S_3\cdot (2,1,1)\cup S_3\cdot (2,2,0).\]
\end{example}

Note that, in this example, we have a bijection ${}^\dagger C_2(3)\cong {}^\dagger C_1(3)\times {}^\dagger C_1(6)$.
More explicitly, this bijection sends every instance of 2 in ${}^\dagger C_1(6)$ to $3\times \frac{2}{2}+0$ in ${}^\dagger C_2(3)$ and every instance of 1 in ${}^\dagger C_1(6)$ to $3\times\frac{1+1}{2}-1-0$.
Below we show that this bijection generalizes. 

\begin{definition}
For each $(n,b)\in \mathbb{Z}^2$, let $\Phi(n,b)\in \mathbb{Z}$ be 
\begin{equation}\label{eq: Phinb}
    \Phi(n,b)=\begin{cases}
        \frac{n}{2}P^{N-1}+b, & n\ \text{even},\\
       \frac{n+1}{2} P^{N-1}-1-b, & n\ \text{odd}.
        \end{cases}
\end{equation}
Let $\Phi^{(3)}: \mathbb{Z}^3\times \mathbb{Z}^3\to \mathbb{Z}^3$ be the function that sends $((n_i)_{i=1,2,3}, (b_i)_{i=1,2,3})$ to $(\Phi(n_i,b_i))_{i=1,2,3}$.
\end{definition}

\begin{lemma}[The local factorization lemma]\label{lm: local fac}
    Assume that $P\geq 3$ is odd and $N\geq 2$, then $\Phi^{(3)}$ restricts to a bijection
    \begin{equation}
        \Phi^{(3)}: {}^\dagger C_1(2P)\times {}^\dagger C_{N-1}(P)\xrightarrow{\sim} {}^\dagger C_N(P).
    \end{equation}
\end{lemma}
\begin{proof}
    Let us first show that the map is well-defined.
    Given $(n_1,n_2,n_3)\in {}^\dagger C_1(2P)$ and $(b_1,b_2,b_3)\in {}^\dagger C_{N-1}(P)$, we need to show that $(\Phi(n_i,b_i))_{i=1,2,3}\in {}^\dagger C_N(P)$.

    Let us check that \eqref{eq: lower} is satisfied for $M<N$.
    In this case, combining \eqref{eq: am} and \eqref{eq: Phinb}, we see that $\Phi(n_i,b_i)_M$ is either $[b_i]_M$ or $P^M-1-[b_i]_M$.
    Therefore, \eqref{eq: lower} is satisfied for $(\Phi(n_i,b_i))_{i=1,2,3}$ since it is satisfied for $(b_1,b_2,b_3)$.

    We now check that \eqref{eq: upper} is also satisfied for $(\Phi(n_i,b_i))_{i=1,2,3}$.
Note that  $2\Phi(n_i,b_i)+1=P^{N-1}n_i+w_i$, where $w_i=2b_i+1$ if $n_i$ is even, and $w_i=P^{N-1}-(2b_i+1)$ if $n_i$ is odd.
By Lemma \ref{eqn: four ineq}, the condition \eqref{eq: upper} for $(\Phi(n_i,b_i))_{i=1,2,3}$ is satisfied if and only if we have 
\begin{align}\label{eq: P^N-1 four ineq}
\begin{split}
    P^{N-1}(2P-2-\sum_{i=1}^3 n_i)+\Big(2P^{N-1}-\sum_{i=1}^3 w_i\Big)&>0,\\
    P^{N-1}(-n_i+\sum_{j\neq i}n_j)+\Big(-w_i+\sum_{j\neq i}w_j\Big)&>0,\quad i=1,2,3.
\end{split}
\end{align}

To continue the argument, we set up the following notation.
Let $H_0=2P^{N-1}-\sum_{i=1}^3(2b_i+1)$, and let $H_i=-(2b_i+1)+\sum_{j\neq i}(2b_j+1)$. 
By Lemma \ref{eqn: four ineq}, the condition that $(b_1,b_2,b_3)\in {}^\dagger C_{N-1}(P)$ is equivalent to the positivity of the $H_i$'s.
We first observe that $H_i<2P^{N-1}$ for $i=0,1,2,3$.  Indeed, $H_0=2P^{N-1}-\sum_{i=1}^3(2b_i+1)<2P^{N-1}$.
For $i=1,2,3$, we have that

\[H_i=2(-b_i+\sum_{j\neq i}b_j)+1\leq 1+2\sum_{j\neq i}b_j\leq 1+2(P^{N-1}-2)<2P^{N-1}.\]

Now consider the four inequalities in \eqref{eq: P^N-1 four ineq}.
All the terms in the first brackets are non-negative, since
$(n_1,n_2,n_3)\in{}^\dagger C_1(2P)$.
Moreover, all four first-bracket terms
have the same parity as $n_1+n_2+n_3$.

The following table records the four second-bracket terms in
\eqref{eq: P^N-1 four ineq}.  In the rows involving $i,j,k$, the indices
$i,j,k$ are distinct elements of $\{1,2,3\}$, and the four entries are listed
in the respective order.
\[
\begin{array}{c|c}
\text{parity pattern} & \text{four second-bracket terms} \\ \hline
n_1,n_2,n_3 \text{ all even}
& (H_0,H_1,H_2,H_3)\\
n_i,n_j \text{ odd and } n_k \text{ even}
& (H_k,H_j,H_i,H_0)\\
n_i \text{ odd and } n_j,n_k \text{ even}
& (P^{N-1}-H_i,\ P^{N-1}-H_0,\ P^{N-1}-H_k,\ P^{N-1}-H_j)\\
n_1,n_2,n_3 \text{ all odd}
& (P^{N-1}-H_0,\ P^{N-1}-H_1,\ P^{N-1}-H_2,\ P^{N-1}-H_3).
\end{array}
\]
Thus, when $n_1+n_2+n_3$ is even, the second-bracket terms are a permutation
of $(H_i)_{i=0,1,2,3}$; when $n_1+n_2+n_3$ is odd, they are a permutation of
$(P^{N-1}-H_i)_{i=0,1,2,3}$.

If $n_1+n_2+n_3$ is even, then every second-bracket term is one of the positive
numbers $H_i$, so all four inequalities in \eqref{eq: P^N-1 four ineq} hold.
If $n_1+n_2+n_3$ is odd, then every first-bracket term is at least $1$, while
every second-bracket term is of the form $P^{N-1}-H_i>-P^{N-1}$, since
$H_i<2P^{N-1}$.  Hence the corresponding left hand side is strictly positive
in this case as well.

Above shows that $\Phi^{(3)}$ sends ${}^\dagger C_1(2P)\times {}^\dagger C_{N-1}(P)$ to ${}^\dagger C_{N}(P)$. We now construct its inverse.

Take an arbitrary element $(a_1,a_2,a_3)\in{}^\dagger C_N(P)$.
  By \eqref{eq: am} and \eqref{eq: lower} for \(M=N-1\), we can
choose
$b_i\in\{[a_i]_{N-1},\,P^{N-1}-1-[a_i]_{N-1}\}$
such that $(b_1,b_2,b_3)$ satisfies \eqref{eq: lower} with $M=N-1$.

We claim that $(b_1,b_2,b_3)\in{}^\dagger C_{N-1}(P)$.  Indeed, it remains
only to check \eqref{eq: am} and \eqref{eq: lower} for $M<N-1$. 
For such $M$,  we have $\{[b_i]_M,\,P^M-1-[b_i]_M\}
=
\{[a_i]_M,\,P^M-1-[a_i]_M\}$.
Therefore the choices that work for $(a_1,a_2,a_3)$ also work for
$(b_1,b_2,b_3)$.

Moreover, the choice of each $b_i$ is unique.  To see this, note that
\eqref{eq: upper} implies that every coordinate of an element of ${}^\dagger C_{N-1}(P)$ is at most $\frac{P^{N-1}-3}{2}$.
Indeed, if $(x_1,x_2,x_3)\in{}^\dagger C_{N-1}(P)$, then
$x_i\leq \sum_{j\neq i}x_j$ and $\sum_i x_i\leq P^{N-1}-2$, hence $2x_i\leq P^{N-1}-2$.  Since $P^{N-1}$ is odd, this gives the desired
bound.  
Therefore, among the two numbers $[a_i]_{N-1}$ and $P^{N-1}-1-[a_i]_{N-1}$,
at most one can be a coordinate of an element of
${}^\dagger C_{N-1}(P)$.

We can now define $n_i$ so that $a_i=\Phi(n_i,b_i)$.
Namely, let
\[
        n_i=
        \begin{cases}
        2\frac{a_i-[a_i]_{N-1}}{P^{N-1}},
        & b_i=[a_i]_{N-1},\\[6pt]
        2\frac{a_i-[a_i]_{N-1}}{P^{N-1}}+1,
        & b_i=P^{N-1}-1-[a_i]_{N-1}.
        \end{cases}
\]

It remains to show that $(n_1,n_2,n_3)\in{}^\dagger C_1(2P)$, i.e., the first brackets in the four inequalities in \eqref{eq: P^N-1 four ineq} are all non-negative.  
The argument is the reverse of the argument for well-definedness.

Since
$(a_1,a_2,a_3)\in{}^\dagger C_N(P)$, Lemma \ref{eqn: four ineq} gives the positivity of the four inequalities in \eqref{eq: P^N-1 four ineq}.  Suppose,
for contradiction, that one of the terms in the first brackets in
\eqref{eq: P^N-1 four ineq} is negative.

If $n_1+n_2+n_3$ is even, then every term in the first brackets is even, so
the negative first-bracket terms are $\leq -2$.  We have already shown that $H_i<2P^{N-1}$.
Thus the corresponding left hand side in \eqref{eq: P^N-1 four ineq} is
strictly negative, a contradiction.

If $n_1+n_2+n_3$ is odd, then every term in the first brackets is odd, so the
negative first-bracket terms are $\leq -1$.  In this case, the corresponding second bracket
is  $P^{N-1}-H_i<P^{N-1}$. Again the corresponding left hand side in
\eqref{eq: P^N-1 four ineq} is strictly negative, a contradiction.

Therefore, we have found the unique $(n_1,n_2,n_3)\in {}^\dagger C_1(2P)$ and $(b_1,b_2,b_3)\in {}^\dagger C_{N-1}(P)$, which defines $\Phi^{(3),-1}(a_1,a_2,a_3)$.
\end{proof}

\subsection{The set $\Ed_{P,N,G}$} \label{subsec: global}\;

For us, a finite graph $G$ consists of finite sets $V(G)$, $E(G)$, and $H(G)$ of
vertices, edges, and half-edges, together with maps
$\nu:H(G)\to V(G)$ and $\epsilon:H(G)\to E(G)$,
such that every fiber of $\epsilon$ has cardinality $2$.
Multiple edges and loops are allowed.
If the two half-edges over an edge $e$ have the same image under $\nu$, then
$e$ is a loop.
For a vertex $v$, the set $H_v=\nu^{-1}(v)$
is the set of branches incident to $v$, counted with multiplicity.
A graph is trivalent if $|H_v|=3$ for every vertex $v$.
In particular, a loop at $v$ contributes two of the three incident branches.
Let $G$ be a finite trivalent graph with edge set $E(G)$.

\begin{definition}\cite[\S5.1]{KW}\label{def: Ed}
Let $\Ed_{P,N,G}$ be the set of functions
 $\ell:E(G)\to \mathbb{Z}_{\geq 0}$
such that, at every vertex $v$ of $G$, the three incident branch labels form a
triple in ${}^\dagger C_N(P)$.
\end{definition}

\begin{example}
    Let $G=K_4$ be the complete graph with 4 vertices.
    Then $\Ed_{3,1,G}$ is the singleton where every edge is numbered 0.
    Direct calculation shows that $|\Ed_{3,2,G}|=|\Ed_{6,1,G}|=49$.
    Moreover, every edge is numbered $0,2$, or $3$ by $\Ed_{3,2,G}$ and every edge is numbered $0,1,$ or $2$ by $\Ed_{6,1,G}$. Changing each $2$ to $3=\Phi(2,0)$ and $1$ to $2=\Phi(1,0)$ defines a bijection $\Ed_{6,1,G}\times \Ed_{3,1,G}\cong \Ed_{3,2,G}$.
    The same idea proves Theorem \ref{thm: level-reduction} as below.
\end{example}

\begin{proof}[Proof of Theorem \ref{thm: level-reduction}]
    By iteration, it suffices to prove the one-step version
\[\Ed_{P,N,G}\cong \Ed_{2P,1,G}\times \Ed_{P,N-1,G}.\]
Indeed, this follows by applying Lemma \ref{lm: local fac} edge-wise on $G$.
The inverse construction in Lemma \ref{lm: local fac} is
coordinatewise.  More explicitly, for an edge $e$, let $b(e)$ be the unique element of
$\{[\ell(e)]_{N-1},\ P^{N-1}-1-[\ell(e)]_{N-1}\}$
which can occur as a coordinate of an element of ${}^\dagger C_{N-1}(P)$;
equivalently, the unique element of this set which is at most $(P^{N-1}-3)/2$.
Then define
\[
        n(e)=
        \begin{cases}
        2\frac{\ell(e)-[\ell(e)]_{N-1}}{P^{N-1}},& b(e)=[\ell(e)]_{N-1},\\[6pt]
        2\frac{\ell(e)-[\ell(e)]_{N-1}}{P^{N-1}}+1,& b(e)=P^{N-1}-1-[\ell(e)]_{N-1}.
        \end{cases}
\] 
Applying Lemma \ref{lm: local fac} at every vertex shows that
$(e\mapsto n(e))\in \Ed_{2P,1,G}$ and
$(e\mapsto b(e))\in \Ed_{P,N-1,G}$.
\end{proof}

\section{Polynomiality}\label{sec: polynomiality}

By \cite[Theorem 5.3]{KW}, the generic degree $\deg(V)$ of the generalized Verschiebung map is given by $\frac{|\Ed_{p,2,G}|}{|\Ed_{p,1,G}|}$, where $G$ is any genus $g$ connected finite trivalent graph.
By Theorem \ref{thm: level-reduction}, we have that 
\begin{equation}\label{eq: V 2p G}
    \deg(V)=|\Ed_{2p,1,G}|.
\end{equation}
The goal of this section is to write out an explicit polynomial of $p$ for $|\Ed_{2p,1,G}|$.
The idea is closely related to Gromov-Witten Theory, and is probably known to the experts.
For example, in \cite[Theorem A]{WakFormula}, Wakabayashi counts dormant opers using the Vafa-Intriligator formula and gives explicit expression for $|\Ed_{p,1,G}|$ in terms of trigonometric sums.
However, we decide to give a streamlined derivation of the formula, assuming only some basic enumerative combinatorics.

We start by taking the graph $G$ to be built up from blocks of the following form, where the edges will be numbered by $a,x,y,b\in \mathbb{Z}_{\geq 0}$ respectively:

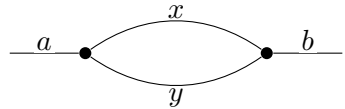
\begin{center}
\centering
\begin{tikzpicture}[
    vertex/.style={circle,fill,inner sep=1.6pt},
    lab/.style={midway,fill=white,inner sep=1pt}
]
\node[vertex] (L) at (0,0) {};
\node[vertex] (R) at (2.4,0) {};

\draw (-1,0) -- node[lab,above] {$a$} (L);
\draw (R) -- node[lab,above] {$b$} (3.4,0);

\draw[bend left=35] (L) to node[lab,above] {$x$} (R);
\draw[bend right=35] (L) to node[lab,below] {$y$} (R);
\end{tikzpicture}
\end{center}
For example, when $g=5$, we take $G$ to be the following graph:
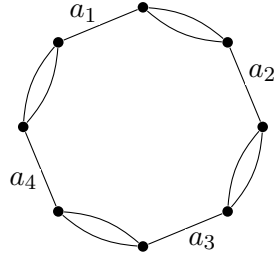
\begin{center}
\centering
\begin{tikzpicture}[
    scale=0.72,
    vertex/.style={circle,fill,inner sep=1.4pt},
    lab/.style={midway,fill=white,inner sep=1pt}
]
\foreach \i in {0,1,2,3} {
    \coordinate (A\i) at ({90-90*\i}:2.2);
    \coordinate (B\i) at ({45-90*\i}:2.2);
    \node[vertex] (L\i) at (A\i) {};
    \node[vertex] (R\i) at (B\i) {};
    \draw[bend left=18] (L\i) to (R\i);
    \draw[bend right=18] (L\i) to (R\i);
}

\draw (R0) -- node[lab,above right] {$a_2$} (L1);
\draw (R1) -- node[lab,below right] {$a_3$} (L2);
\draw (R2) -- node[lab,below left] {$a_4$} (L3);
\draw (R3) -- node[lab,above left] {$a_1$} (L0);
\end{tikzpicture}
\end{center}

For each block, and fixed numberings $0\leq a,b\leq p-1$, let
\[
T_{ab}=|\{(x,y):\; (a,x,y),(b,x,y)\in {}^\dagger C_1(2p)\}|
\]
be the number of ways to choose $x$ and $y$ so that this block obeys the conditions of $\Ed_{2p,1,G}$.  
Let $T=(T_{ab})_{0\leq a,b\leq p-1}$
be the corresponding $p\times p$ matrix.  Then
\begin{equation}\label{eqn: trace}
    |\Ed_{2p,1,G}|
=
\sum_{a_1,\ldots,a_{g-1}=0}^{p-1}
T_{a_1a_2}T_{a_2a_3}\cdots T_{a_{g-1}a_1}
=
\operatorname{Tr}(T^{g-1}).
\end{equation}
The last equality follows from some facts of transfer matrices in enumerative combinatorics as in \cite[\S4.7.1]{StanleyEC1}.  Indeed, consider
the auxiliary weighted directed graph with vertex set $\{0,\ldots,p-1\}$, in
which the edge $a\to b$ has weight $T_{ab}$.  Then the above summation is
the total weight of closed walks of length $g-1$, which is
$\operatorname{Tr}(T^{g-1})$ by \cite[first line of the proof of Corollary 4.7.3]{StanleyEC1}

We now find $\operatorname{Tr}(T^{g-1})$ by diagonalizing $T$.
For this purpose, we introduce the following notation:
\[
    N_{ab}^{c}=
    \begin{cases}
    1,& (a,b,c)\in{}^\dagger C_1(2p),\\
    0,& \text{otherwise}.
    \end{cases}
\]
Let us remark that this looks very similar to the even part of the fusion ring as in \cite[(1.1)]{RunkelSuszek}. 

Observe that 
\begin{equation}\label{eqn: tnn}
    T_{ab}=\sum_{x,y=0}^{p-1}N_{ax}^y N_{bx}^y.
\end{equation}
The next lemma turns $N_{ab}^c$ into a trigonometric sum.

\begin{lemma}\label{lm: trig}
For \(1\leq j\leq p\), set
\[
S_{aj}=
\begin{cases}
\sqrt{\frac{2}{p}}\,
\sin\!\left(\frac{(2a+1)j\pi}{2p}\right),
& 1\leq j\leq p-1,\\[6pt]
\frac{(-1)^a}{\sqrt p},
& j=p.
\end{cases}
\]
Then,
\begin{enumerate}
    \item[(a)] The matrix $S=(S_{aj})_{0\leq a\leq p-1,\ 1\leq j\leq p}$
is orthogonal.
\item[(b)] \begin{equation}\label{eq: finite-verlinde}
    N_{ab}^{c}
    =
    \sum_{j=1}^{p}
    \frac{S_{aj}S_{bj}S_{cj}}{S_{0j}}.
\end{equation}
\end{enumerate}

\end{lemma}
\begin{proof}
Set $\theta_j=\frac{j\pi}{2p}$

(a)
Since $2\sin X\sin Y=\cos(X-Y)-\cos(X+Y)$, we have that, for $1\leq j,k\leq p-1$,
\[
\sum_{a=0}^{p-1}S_{aj}S_{ak}
=
\frac1p
\sum_{a=0}^{p-1}
\left[
\cos((2a+1)(\theta_j-\theta_k))
-
\cos((2a+1)(\theta_j+\theta_k))
\right].
\]
Now use the arithmetic progression cosine formula
$\sum_{a=0}^{p-1}\cos((2a+1)u)=\frac{\sin(2pu)}{2\sin u}$,
whenever $\sin u\neq 0$. In our case, $u=(j-k)\pi/(2p)$, or $u=(j+k)\pi/(2p)$. Therefore, the sum vanishes unless $j=k$. 
If $j=k$, the first cosine sum is $p$, and the second still vanishes. Therefore, $\sum_{a=0}^{p-1}S_{aj}S_{ak}=\delta_{jk}$.
For the last column, we have that $S_{a,p}=\frac{(-1)^a}{\sqrt p}
=
\frac{1}{\sqrt p}\sin((2a+1)\pi/2)$.
The same product-to-sum argument gives $\sum_{a=0}^{p-1}S_{a,p}S_{a,k}=0
\qquad (1\leq k\leq p-1)$,
and clearly $\sum_{a=0}^{p-1}S_{a,p}^2=1$.

(b). 
We claim that for every $1\leq j\leq p$,
\begin{equation}
    \label{eqn: key trig}
    \sum_{c=0}^{p-1}N_{ab}^{c}\sin((2c+1)\theta_j)
=
\frac{
\sin((2a+1)\theta_j)\sin((2b+1)\theta_j)
}{
\sin\theta_j
}.
\end{equation}

Because the condition is symmetric in \(a,b\), assume \(a\geq b\).

First suppose \(a+b\leq p-1\). Then the possible values for $c$ are $a-b,\ a-b+1,\ldots,a+b$, so that \[\sum_c \sin((2c+1)\theta_j)
=
\sum_{r=0}^{2b}
\sin((2(a-b+r)+1)\theta_j)=\frac{
\sin((2b+1)\theta_j)\sin((2a+1)\theta_j)
}{
\sin\theta_j
},\]
where we have used the arithmetic progression sine formula for the last equality.

If $a+b\geq p$, then the possible values for $c$ are $a-b,\ a-b+1,\ldots,2p-2-a-b$, and we can conclude using the same idea as above.

Now \eqref{eqn: key trig} give us
$\sum_{c=0}^{p-1}N_{ab}^{c}S_{cj}
=
\frac{S_{aj}S_{bj}}{S_{0j}}$ for every $1\leq j\leq p$. 
When $j=p$, the same formula becomes $\sum_{c} N_{ab}^c(-1)^c=(-1)^{a+b}$, which is also true.
Indeed, from above we see that the admissible
values of $c$ form an interval of odd length starting with $a-b$, hence $\sum_c (-1)^c=(-1)^{a-b}=(-1)^{a+b}$.

Finally, we can apply the orthogonality of $S$ to conclude (b).
\end{proof}

\begin{proof}[Proof of Theorem \ref{thm: polynomiality}]
Combining \eqref{eqn: tnn} and \eqref{eq: finite-verlinde}, we get
\[
\begin{aligned}
T_{ab}
&=
\sum_{x,y=0}^{p-1}N_{ax}^{y}N_{bx}^{y} =
\sum_{x,y=0}^{p-1}
\left(\sum_{j=1}^{p}\frac{S_{aj}S_{xj}S_{yj}}{S_{0j}}\right)
\left(\sum_{k=1}^{p}\frac{S_{bk}S_{xk}S_{yk}}{S_{0k}}\right) \\
&=
\sum_{j,k=1}^{p}
\frac{S_{aj}S_{bk}}{S_{0j}S_{0k}}
\left(\sum_{x=0}^{p-1}S_{xj}S_{xk}\right)
\left(\sum_{y=0}^{p-1}S_{yj}S_{yk}\right).
\end{aligned}
\]
By Lemma \ref{lm: trig}.(a), both inner sums are equal to $\delta_{jk}$, hence
$T_{ab}
=
\sum_{j=1}^p \frac{S_{aj}S_{bj}}{S_{0j}^2}$.

 Therefore, if we set $D=\operatorname{diag}\left(S_{01}^{-2},\ldots,S_{0p}^{-2}\right)$,
then
$T=SDS^t$, so that $\operatorname{Tr}(T^{g-1})
=
\sum_{j=1}^p S_{0j}^{-2(g-1)}$.
Combining with \eqref{eq: V 2p G} and \eqref{eqn: trace},
we obtain
\begin{equation}\label{eqn: deg trig}
\deg(V)
=
p^{g-1}
+
\left(\frac p2\right)^{g-1}
\sum_{j=1}^{p-1}
\csc^{2g-2}\left(\frac{\pi j}{2p}\right).
\end{equation}

By symmetry about $p$,
\[
\sum_{j=1}^{2p-1}
\csc^{2g-2}\left(\frac{\pi j}{2p}\right)
=
2\sum_{j=1}^{p-1}
\csc^{2g-2}\left(\frac{\pi j}{2p}\right)+1.
\]
Hence
\[
\deg(V)
=
p^{g-1}
+
\frac{p^{g-1}}{2^g}
\left(
\sum_{j=1}^{2p-1}
\csc^{2g-2}\left(\frac{\pi j}{2p}\right)-1
\right).
\]

Now the theorem follows from Zagier's formula \cite[Theorem 1(i),(iii)]{Zagier}:
\[
\sum_{j=1}^{2p-1}
\csc^{2g-2}\left(\frac{\pi j}{2p}\right)
=
\sum_{r=0}^{g-1}
\frac{(-1)^{r-1}2^{4r}B_{2r}}{(2r)!}
\left[z^{-2r}\right]
\csc^{2g-2}(z)
\,p^{2r},
\]
where $\left[z^{-2r}\right]
\csc^{2g-2}(z)$ means the coefficient of $z^{-2r}$ in the Laurent expansion of $\csc^{2g-2}(z)$ at $z=0$.
\end{proof}

\section{Dormant opers}\label{sec: oper}

\subsubsection*{Opers}\;

Let $X$ be a smooth projective curve over an algebraically closed field $k$ of genus $g>1$.
Let $\omega_X$ be the canonical sheaf of $X$.

In this paper, by opers, we mean what are usually called $\operatorname{PGL}_2$-opers.
Namely, an oper over $X$ is a $\mathbb{P}^1$-bundle $\pi: P\to X$ together with a connection $\nabla$, i.e., a splitting of the tangent exact sequence 
\[
0\longrightarrow T_{P/X}
\longrightarrow T_P
\mathrel{\substack{\xrightarrow{\ d\pi\ }\\[-0.6ex]\xleftarrow[\ \nabla\ ]{}}}
\pi^*T_X
\longrightarrow 0,
\]such that there is a section $\sigma: X\to P$ so that we have an induced isomorphism $d\sigma-\sigma^*\nabla:T_X\to \sigma^*T_{P/X}$.

Note that the choice of $\sigma$ is unique up to isomorphism if it exists.
Indeed, if $\sigma'$ is another such section, then $\sigma'$ determines a nonzero section of the bundle $\sigma^*T_{P/X}$, contradicting the negativity of the degree of $\sigma^*T_{P/X}\cong T_X$.

\subsubsection*{Uniformization}\;

As explained in \cite[\S1]{MocIntro}, one
 motivation for the introduction of opers is that, over $\mathbb{C}$, every $X$ has a canonical oper over it. Indeed, let $\mathbb{H}\subset \mathbb{C}$ be the upper half plane and $\pi_1(X)\to \operatorname{Aut}(\mathbb{H})\subseteq \operatorname{Aut}(\mathbb{P}^1_{\mathbb{C}})$ be the representation associated to the universal cover of $X$.
Then the $\mathbb{P}^1$-bundle on $X$ given by the quotient of $\mathbb{H}\times \mathbb{P}^1_{\mathbb{C}}$ by the diagonal action of $\pi_1(X)$ carries a canonical oper structure: the connection is given by descending the trivial connection on $\mathbb{H}\times \mathbb{P}^1_{\mathbb{C}}$ and the section $\sigma$ is given by descending the tautological section $\mathbb{H}\to \mathbb{H}\times \mathbb{P}^1_{\mathbb{C}}$.

One is interested in studying how special the canonical oper is. For example, the associated representation $\pi_1(X)\to \operatorname{PGL}_2(\mathbb{C})$ is real-valued.
A natural way to imitate this property in the $p$-adic setting is to require the oper to admit some kind of Frobenius structure.
In turn, in characteristic $p$, this requirement boils down to the nilpotency of the $p$-curvature.
Among the opers with nilpotent $p$-curvatures, the ones with \textit{zero} $p$-curvatures have been particularly studied.

\subsubsection*{Dormant opers}\;

An oper $(P,\nabla)$ on $X$ is called dormant if its $p$-curvature $\psi(\nabla)$ vanishes.
Equivalently, by Cartier descent and the \'etale local triviality of $\operatorname{PGL}$-torsors, this means that the pair $(P,\nabla)$ is \'etale locally isomorphic to the constant $\mathbb{P}^1$-bundle together with the canonical connection.

\subsubsection*{Over $W_N(k)$}\;

Let $W_N(k)$ be the $N$-th ring of Witt vectors.
The definition of opers as above also makes sense for a lift $\widetilde{X}$ over $W_N(k)$ of $X$.
The dormancy of an oper over $\widetilde{X}$ is treated by Wakabayashi in \cite{WakTQFT}.
The definition is somewhat involved, but it is shown in \cite[Corollary 3.12]{WakTQFT} to be equivalent to the \'etale-local constancy condition $(P,\nabla_P)\cong (\mathbb{P}^1\times \widetilde{X},\nabla^{\mathrm{can}})$ as above.

\subsubsection*{Counting dormant opers}\;

It is shown by Wakabayashi in \cite[Theorem D, Theorem E]{WakTQFT} that, for a general $X$ with a lift $\widetilde{X}$ over $W_N(k)$, the number $|\Ed_{p,N,G}|$ is the same as the number of dormant opers over $\widetilde{X}$.
Therefore, the oper-theoretic corollary of our Theorem \ref{thm: level-reduction} is:
\begin{corollary}\label{cor: counts}
    Let $X$ be a general curve of genus $g>1$ over an algebraically closed field $k$ together with a lift $\widetilde{X}$ over $W_N(k)$.
    Let $\operatorname{Op}^Z_X$ and $\operatorname{Op}_{\widetilde{X}}^Z$ be the set of isomorphism classes of dormant opers over $X$ and $\widetilde{X}$ respectively.
    Then, we have the following equality:
    \[|\operatorname{Op}_{\widetilde{X}}^Z|=|\operatorname{Op}^Z_X|\cdot \deg(V)^{N-1}.\]
\end{corollary}

\;\\
\;\\

\subsection*{Acknowledgements}
I would like to thank Michel Gros, who,
during the 2023-2024 special year program on $p$-adic geometry at IAS, introduced me to the works of Wakabayashi.
I would also like to thank Junliang Shen: recently some  discussions with Junliang led me to think about the generalized Verschiebung map.
This work is supported by an AMS-Simons travel grant.

\;\\

\;\\

\footnotesize{
 \textsc{Department of Mathematics, Yale University,  New Haven, CT, 06511,
USA}\par\nopagebreak
  \textit{E-mail address}: \texttt{siqing.zhang.math@gmail.com}}

\end{document}